\documentclass[11pt,letterpaper]{amsart}

\usepackage{amsmath}
\usepackage{amsthm}
\usepackage{amssymb}
\usepackage{amscd}
\usepackage{amsfonts}
\usepackage{amsbsy}
\usepackage[noadjust]{cite} 
\usepackage{epsfig,afterpage}
\usepackage{paralist}
\usepackage{psfrag}
\usepackage{mathrsfs}       
\usepackage{verbatim}

\newtheorem {theorem} {Theorem}

\newtheorem {proposition} [theorem]{Proposition}

\newtheorem {lemma}  [theorem]{Lemma}
\newtheorem {remark} [theorem]{\sc Remark}

\title[Principal Bautin ideal of some monodromic singularities]{Principal Bautin ideal of monodromic singularities with inverse integrating factors}

\thanks{The authors are partially supported by Agencia Estatal de Investigaci\'on grant number PID2020-113758GB-I00 and an AGAUR grant number 2021SGR-01618.}

\author[Isaac A. Garc\'ia and Jaume Gin\'e]{Isaac A. Garc\'ia and Jaume Gin\'e}

\address{Departament de Matem\`atica, Universitat de Lleida, Avda. Jaume II, 69, 25001 Lleida, Spain}

\email{isaac.garcia@udl.cat}
\email{jaume.gine@udl.cat}



\begin{document}

\begin{abstract}
We analyze the structure of the Poincar\'e map $\Pi$ associated to a monodromic singularity of an analytic family of planar vector fields. We work under two assumptions. The first one is that the family possesses an inverse integrating factor that can be expanded in Laurent series centered at the singularity after a weighted polar blow-up fixed by the Newton diagram of the family. The second one is that we restrict our analysis to a subset of the monodromic parameter space that assures the non-existence of local curves with zero angular speed. The conclusions are that the asymptotic Dulac expansion of $\Pi$ does not contain logarithmic terms, indeed it admits a formal power series expansion with a unique independent generalized Poincar\'e-Lyapunov quantity, which can be computed under some explicit conditions. Moreover we also give conditions that guarantee the analyticity of $\Pi$, in which case we show that the Bautin ideal is principal and therefore the cyclicity of the singularity with respect to perturbation within the family is zero.
\end{abstract}

\maketitle

\section{Introduction and main results}

We consider families of real analytic planar differential systems
\begin{equation}\label{DCC-1}
\dot{x}= P(x,y; \lambda),  \ \ \dot{y}=Q(x,y; \lambda),
\end{equation}
depending analytically on the parameters $\lambda \in \mathbb{R}^p$. We associate to \eqref{DCC-1} the family of vector fields $\mathcal{X} = P(x,y; \lambda) \partial_x + Q(x,y; \lambda) \partial_y$. Throughout the work we restrict the family to the parameter space $\Lambda \subset \mathbb{R}^p$ so that the origin $(x,y)=(0,0)$ is made sure to be a {\it monodromic singularity} of the whole family \eqref{DCC-1}. That means that $P(0, 0; \lambda) = Q(0, 0; \lambda) = 0$ and the local orbits of $\mathcal{X}$ turn around the origin for any $\lambda \in \Lambda$.
\newline

Let $(p,q) \in W(\mathbf{N}(\mathcal{X}))$ be two weights associated to the Newton diagram $\mathbf{N}(\mathcal{X})$ of $\mathcal{X}$, see subsection \ref{Sec-IB2} for the definition of $\mathbf{N}(\mathcal{X})$. One we pick up $(p,q)$ we obtain an expansion
\begin{equation}\label{campo-X}
\mathcal{X} = \sum_{j \geq r} \mathcal{X}_j
\end{equation}
where $r \geq 1$ and $\mathcal{X}_j$ are $(p,q)$-quasihomogeneous vector fields of degree $j$. We call $\mathcal{X}_r$ the leading part of $\mathcal{X}$.

We perform the {\it weighted polar} blow-up $(x,y) \mapsto (\rho, \varphi)$ given by $(x,y) = (\rho^p \cos\varphi, \rho^q \sin\varphi)$ transforming \eqref{DCC-1} into the polar vector field $\dot{\rho} = R(\varphi, \rho)$, $\dot{\varphi} = \Theta(\varphi, \rho)$ and consider the differential equation
\begin{equation}\label{eq3}
\frac{d\rho}{d\varphi} \, = \, \mathcal{F}(\varphi, \rho),
\end{equation}
where $\mathcal{F}(\varphi, \rho) = R(\varphi, \rho)/\Theta(\varphi, \rho)$ is a function well defined in $C \backslash \Theta^{-1}(0)$ being the cylinder
\begin{equation}\label{cylinder}
C  \, = \, \left\{ (\varphi, \rho) \in \mathbb{S}^1 \times \mathbb{R} \, : \, 0 \leq \rho \ll 1 \right\}, \ \mbox{with } \mathbb{S}^1 = \mathbb{R}/ (2 \pi \mathbb{Z}),
\end{equation}
and the zero-set
\begin{equation}\label{crit-set}
\Theta^{-1}(0) = \{ (\varphi, \rho) \in C : \Theta(\varphi, \rho) = 0, \rho \geq 0 \}.
\end{equation}
We define the set of {\it characteristic directions} $\Omega_{pq} = \{ \varphi^* \in \mathbb{S}^1 : \Theta(\varphi^*, 0)  = 0 \}$. Since $R(\varphi, 0) \equiv 0$ the set $\{\rho=0 \}$ is invariant for the flow of \eqref{eq3} and it becomes either a periodic orbit or a polycycle according to whether $\Omega_{pq} = \emptyset$ or $\Omega_{pq} \neq \emptyset$, respectively.
\newline

A not locally null real-valued $C^1(C \backslash \Theta^{-1}(0))$ function $V(\varphi, \rho)$ is an {\it inverse integrating factor} of \eqref{eq3} in $C \backslash \Theta^{-1}(0)$ if it is a solution of the linear partial differential equation
\begin{equation}\label{EDP-V}
\frac{\partial V}{\partial \varphi}(\varphi, \rho) + \frac{\partial V}{\partial \rho}(\varphi, \rho) \, \mathcal{F}(\varphi, \rho)  =
\frac{\partial \mathcal{F}}{\partial \rho}(\varphi, \rho) \, V(\varphi, \rho).
\end{equation}

\begin{remark}\label{rem-v-V}
{\rm If $v(x,y)$ is an inverse integrating factor of \eqref{DCC-1}, that is, the differential 1-form $(P(x,y; \lambda) \, dy - Q(x,y; \lambda) \, dx) / v(x,y)$ is closed in a neighborhood of the origin except at the zero-set $v^{-1}(0)$, then
\begin{equation}\label{V-polar-def}
V(\varphi, \rho) = \frac{v(\rho^p \cos\varphi, \rho^q \sin\varphi)}{\rho^{r} \, J(\varphi, \rho) \, \Theta(\varphi, \rho)}
\end{equation}
is an inverse integrating factor of \eqref{eq3} in $C \backslash \{\Theta^{-1}(0) \cup \{\rho=0\} \}$ where $J(\varphi, \rho) = \rho^{p+q-1}(p \cos^2\varphi + q \sin^2\varphi)$ is the Jacobian of the weighted polar blow-up (which only vanishes at $\rho=0$) and $r \in \mathbb{N}$ is the $(p,q)$--quasihomogeneous degree of the leading vector field associated to \eqref{DCC-1}, see \eqref{campo-X}.}
\end{remark}

We define {\it Puiseux inverse integrating factors} as those inverse integrating factors of \eqref{eq3} in $C \backslash \{\Theta^{-1}(0) \cup \{\rho=0\} \}$ that can be expanded in convergent Puiseux series about $\rho=0$ of the form
\begin{equation}\label{V-Puiseux-cil}
V(\varphi, \rho) = \sum_{i \geq m} v_i(\varphi) \rho^{i/n}
\end{equation}
whose coefficients, for a fixed $\lambda$, are analytic functions $v_i : \mathbb{S}^1 \backslash \Omega_{pq} \to \mathbb{R}$, the leading coefficient $v_m(\varphi; \lambda) \not\equiv 0$, and $(m,n) \in \mathbb{Z} \times \mathbb{N}^*$ are fixed numbers called {\it multiplicity} and {\it index}, respectively. Here we have used the notation $\mathbb{N}^* = \mathbb{N} \backslash \{0\}$. The particular case $n=1$ leads to a Laurent inverse integrating factor while $n=1$ and $m \geq 0$ produces an formal inverse integrating factors.
\newline

We are going to see that, after a change in the radial variable we can take $n=1$ in \eqref{V-Puiseux-cil} without lost of generality.

\begin{lemma}\label{V->puiseux-laurent}
Any Puiseux inverse integrating factor $V(\varphi, \rho)$ of the form \eqref{V-Puiseux-cil} with multiplicity and index $(m, n)$ can be transformed into a Laurent inverse integrating factor $\tilde{V}(\varphi, \sigma)$ via the change $\sigma^n = \rho$ whose explicit expression is $\tilde{V}(\varphi, \sigma) = V(\varphi, \sigma^n)/(n \sigma^{n-1})$ with multiplicity $m-n+1$.
\end{lemma}

After Lemma \ref{V->puiseux-laurent} it is natural that we focus our attention on inverse integrating factors $V(\varphi, \rho)$ of \eqref{eq3} defined in $C \backslash \Theta^{-1}(0)$ that can be expanded in Laurent power series about $\rho=0$ of the form
\begin{equation}\label{V-Cw-def}
V(\varphi, \rho) = \sum_{i \geq m} v_i(\varphi) \rho^{i},
\end{equation}
for some leading exponent given by the  multiplicity $m \in \mathbb{Z}$. The coefficients of this expansion are $C^1$ functions $v_i : \mathbb{S}^1 \backslash \Omega_{pq} \to \mathbb{R}$ and the leading coefficient $v_m(\varphi) \not\equiv 0$.

Given the weights $(p,q) \in W(\mathbf{N}(\mathcal{X}))$, we define the $(p,q)$-{\it critical parameters} as the elements of the subset $\Lambda_{pq} \subset \Lambda$ of the monodromic parameter space corresponding to vector fields with local curves of zero angular speed, that is,
\begin{equation}\label{critical-L}
\Lambda_{pq} = \{ \lambda \in \Lambda : \Theta^{-1}(0) \backslash \{ \rho = 0\} \neq \emptyset \}.
\end{equation}
We emphasize that $\Lambda_{pq} = \emptyset$ when $\Omega_{pq} = \emptyset$ but the converse is not true.

The Poincar\'e map $\Pi: \Sigma \subset (\mathbb{R}^+, 0) \to (\mathbb{R}^+, 0)$ associated to the monodromic singularity at the origin of \eqref{DCC-1} is defined in the transversal section $\Sigma = \{ (x, 0) \in \mathbb{R}^2 : 0 < x \ll 1 \}$.  From Il'Yashenko's work \cite{Ilyashenko} we know that $\Pi$ can be non-differentiable at the origin but it is a semiregular map that possesses a Dulac asymptotic
expansion $\Pi(x) = \eta_1 x + o(x)$ with linear leading coefficient $\eta_1 > 0$. Using the analyticity of $\mathcal{X}$ in $\lambda$, Medvedeva in \cite{Me} proves that $\Pi$ has a Dulac asymptotic expansion of the form
\begin{equation}\label{Poinc-expan}
\Pi(x) = \eta_1 x + \sum_j P_j(\log x) x^{\nu_j},
\end{equation}
where the exponents $\nu_j > 1$ are independent of $\lambda$ and grow to infinity and the coefficients of the $P_j$ are polynomials whose coefficients depend analytically on the coefficients of $\mathcal{X}$.

We consider the solution $\Phi(\varphi; \rho_0)$ of the Cauchy problem (\ref{eq3}) with initial condition $\Phi(0; \rho_0) = \rho_0 > 0$ sufficiently small. Then we reparameterize $\Sigma$ by $x = \rho_0^{p}$ so that $\Pi(\rho_0) = \Phi(2 \pi, \rho_0)$.

\subsection{The structure of the formal Poincar\'e map}

Adapting Theorem 3 of the work \cite{GaGiGr} to our context we will prove a fundamental relation between the Poincar\'e map $\Pi$ of the monodromic singularity and the inverse integrating factors $V$ of equation \eqref{eq3}. We want to emphasize that, using our notation, the framework in \cite{GaGiGr} was that $V \in C^1(C \backslash \{\rho=0\})$ and that $\{\rho=0\}$ is a periodic orbit, equivalently $\Omega_{pq} = \emptyset$ and $\Lambda_{pq} = \emptyset$. We are going to relax these restrictions allowing that $\{\rho=0\}$ be a polycycle, that is, $\Omega_{pq} \neq \emptyset$ but we keep $\Theta^{-1}(0) \backslash \{\rho=0\} = \emptyset$.

\begin{theorem} \label{teo-Fund-VPi}
In the restricted parameter space $\Lambda \backslash \Lambda_{pq}$ the fundamental equation
\begin{equation}\label{fundamental}
V(0, \Pi(\rho_0)) = V(0, \rho_0) \, \Pi'(\rho_0)
\end{equation}
holds.
\end{theorem}

In the proof of Theorem \ref{lemma-fundamental} we will use some results of the (two-dimensional) power geometry developed by Bruno in \cite{Bru, Br} that is designed to classify asymptotic expansions (including Puiseux series as particular case) of invariant branches at singularities of analytic planar vector fields. We summarize these results in the subsection \ref{Sec-IB2} of the Appendix.

It is nice to remark that in the proof of the following theorem we re-prove, under the conditions of the theorem, the well-known fact that the Poincar\'e map $\Pi$ has an asymptotic expansion of the form $\Pi(\rho) = \eta_1 \rho + o(\rho)$.

\begin{theorem} \label{lemma-fundamental}
Let the origin be a monodromic singularity of family \eqref{DCC-1} with parameters in $\Lambda$ and $0 \not\in \Omega_{pq}$. Assume that equation {\rm (\ref{eq3})} has a Laurent inverse integrating factor $V(\varphi, \rho)$ of the form  \eqref{V-Cw-def} with multiplicity $m$. Then the Poincar\'e map $\Pi$ of the singularity at the origin of the family restricted to the parameter space $\Lambda \backslash \Lambda_{pq}$ has a formal power series expansion $\Pi(\rho) = \sum_{i \geq 1} \eta_i \rho^i$. Moreover $\Pi$ has the following structure:
\begin{itemize}
  \item[(i)] If $m \leq 0$ then the origin is a center.

  \item[(ii)] If $m > 1$ then $\Pi(\rho) = \rho + \eta_m \rho^m (1+ O(\rho))$. In particular, if $\eta_m = 0$ then the origin is a center.

  \item[(iii)] If $m=1$ and $\eta_1 = 1$ then the origin is a center.
\end{itemize}
\end{theorem}

\begin{remark}
{\rm Statement (i) of Theorem \ref{lemma-fundamental} is included for the sake of completeness and it was proved in \cite{GaGi2} for parameters in $\Lambda$ and not only in the more restrictive subset $\Lambda \backslash \Lambda_{pq}$ as it is stated here. On the other hand, we also note that condition $0 \not\in \Omega_{pq}$ of Theorem \ref{lemma-fundamental} is not restrictive because if $0 \in \Omega_{pq}$ then after a rotation of coordinates (or just by the linear change that interchanges the coordinates $x$ and $y$) we get $0 \not\in \Omega_{pq}$. Recall that a change of coordinates may modify the Newton diagram of the vector field.}
\end{remark}

\begin{lemma}\label{structure-main-2-new}
Let $V(\varphi, \rho)$ be a Puiseux inverse integrating factor of \eqref{eq3}. If the origin is a monodromic singularity of system \eqref{DCC-1}, then $V^{-1}(0) \backslash \{\rho=0\} = \emptyset$.
\end{lemma}

Taking $C = I \times \mathbb{S}^1$ the cylinder \eqref{cylinder} with $I = \{0\} \cup I^+$ and $I^+$ is a sufficiently small positive half-neighborhood of the origin, the function $G : I^+ \to \mathbb{R}$ given by
\begin{equation}\label{Function-G}
G(r) = \int_0^{2 \pi} \frac{\mathcal{F}(\varphi, r)}{V(\varphi, r)} d \varphi
\end{equation}
is well defined in the following two cases: (i) Family \eqref{DCC-1} is restricted to $\Lambda \backslash \Lambda_{pq}$; (ii) There is a generalized Darboux inverse integrating factor of \eqref{DCC-1} of the form $v(x,y) = \prod_i f_i^{\lambda_i}(x,y)$ where $f_i$ are analytic functions in a neighborhood of the origin and $\lambda_i \in \mathbb{Q}$.

\begin{theorem} \label{lemma-fundamental-2}
Under the conditions of Theorem \ref{lemma-fundamental}, $G(r) = \mathfrak{g} \in \mathbb{R}$ is a constant in $I^+$ and $\log\eta_1 = \mathfrak{g}$ when $m=1$ and $\eta_m = \mathfrak{g}$ if $m > 1$.
\end{theorem}

\begin{remark}
{\rm Taking a look at the second part of the proof of Theorem \ref{lemma-fundamental-2} we infer that Theorems \ref{teo-Fund-VPi}, \ref{lemma-fundamental} and \ref{lemma-fundamental-2} can be generalized changing the condition of the restriction to $\Lambda \backslash \Lambda_{pq}$ to the condition that $G$ is well defined in $I^+$ and it is a constant. Using this idea we see, for example, that the restriction $\Lambda \backslash \Lambda_{pq}$ is not necessary in Proposition 22 of \cite{GaGi2}. }
\end{remark}

\begin{remark}
{\rm Although $\mathfrak{g} = \lim_{r \to 0^+} G(r)$ we remark that when $\Omega_{pq} \neq \emptyset$ then, in general,
$$
\lim_{r \to 0^+} G(r) \neq \int_0^{2 \pi} \lim_{r \to 0^+} \frac{\mathcal{F}(\varphi, r)}{V(\varphi, r)} d \varphi.
$$
We refer the reader to Example 2 in \cite{GaGi2} where $G(r)$ is computed and these phenomena appears. We also emphasize that, in general, $G$ cannot be continuously extended to the origin. The reason is that, using \eqref{Function-G}, and the expansions \eqref{V-Cw-def} and $\mathcal{F}(\varphi, r) = \mathcal{F}_1(\varphi) r + O(r^2)$, we see that $G(0)$ does not exist when $m> 1$ whereas for $m=1$ one gets $G(0) = \xi_{pq}$ where
\begin{equation}\label{xi-pq}
\xi_{pq} = PV \int_0^{2 \pi} \mathcal{F}_1(\varphi) d \varphi,
\end{equation}
which may exists or not and even when exists it can be different from $\mathfrak{g}$. }
\end{remark}

\subsection{The Bautin ideal and the cyclicity of the monodromic singularity}

Let $E \subset \Lambda$ be an open subset of the monodromic parameter space $\Lambda$ such that the rectricted family $\mathcal{X}|_E$ satisfies that its Newton diagram $\mathbf{N}(\mathcal{X}|_E)$ is fixed. For any $\lambda^* \in E$ we consider the family of vector fields for $(x,y, \lambda)$ on an open neighborhood of $((0, 0), \lambda^*)$ in $\mathbb{R}^2 \times E$. Then for $|\rho|$ and $\|\lambda - \lambda^* \|$ sufficiently small, $\Pi$ can be expressed as a formal power series $\Pi(\rho; \lambda) = \sum_{i \ge 1} \eta_i(\lambda) \rho^i$, see Theorem \ref{lemma-fundamental}. The generalized Poincar\'e-Lyapunov quantities $\eta_i(\lambda)$ are analytic on a neighborhood of $\lambda^*$. This assertion was proved by Medvedeva \cite{Me} even when the Poincar\'e map has a Dulac asymptotic expansion. In the forthcoming example \eqref{Ejemplo-V-foco-ciclo} we show that the $\eta_i(\lambda)$ can be even discontinuous at $\lambda = \hat\lambda \in \partial E$, the boundary of $E$.

We identify each $\eta_i(\lambda)$ as an element of the ring $\mathcal{G}_{\lambda^*}$ of germs of analytic functions at $\lambda^*$.
Since $\mathcal{G}_{\lambda^*}$ is a noetherian ring, the {\it Bautin ideal} defined as $\mathcal{B} = \langle \eta_1-1, \eta_2, \eta_3, \ldots, \rangle$ so generated by $\eta_1-1$ and all the $\eta_j$ with $j \geq 2$ is finitely generated.

The Bautin ideal has a narrow relationship with the \emph{cyclicity} of the monodromic singularity at the origin, that is, with the maximum number of small-amplitude limit cycles that can be made to bifurcate from it under small perturbation of relevant parameters in $E$.

Theorem \ref{lemma-fundamental} provides us with a structure of the Bautin ideal and the next result use it to compute the cyclicity.

\begin{theorem} \label{cor-fundamental}
Under the assumptions of Theorem \ref{lemma-fundamental} and restricting the family to $\mathcal{X}|_E$ where $E \subset \Lambda$ is such that the Newton diagram $\mathbf{N}(\mathcal{X}|_E)$ is fixed, the Bautin ideal $\mathcal{B}$ is principal and given by $\mathcal{B} = \langle \eta_1(\lambda)-1 \rangle$ if $m=1$ or $\mathcal{B} = \langle \eta_m(\lambda) \rangle$ when $m \geq 2$. Moreover, the following holds:
\begin{itemize}
  \item[(i)] If $m>1$ and the residue ${\rm Res}(1/V(0, \rho), 0) = 0$ then $\Pi(\rho)$ is analytic at $\rho=0$.
  \item[(ii)] When $\Pi$ is analytic at $\rho=0$ then the cyclicity of the origin, with respect to perturbation within the family $\mathcal{X}|_E$, is 0.
  \item[(iii)] If $m=1$ then $\Pi(\rho)$ is analytic.
\end{itemize}
\end{theorem}

\subsection{Non-continuous dependence of $\Pi$ with respect to the parameters when $\mathbf{N}(\mathcal{X})$ changes}
We emphasize that keeping the family in $E$ is necessary in Corollary \ref{cor-fundamental} as the following example shows. We consider the 1-parameter family
\begin{equation}\label{Ejemplo-V-foco-ciclo}
\dot{x} = (x - y)(x^2 + y^2) - \lambda (x + y), \ \ \dot{y} = (x + y)(x^2 + y^2) + \lambda (x - y).
\end{equation}
It has the inverse integrating factor $v(x,y; \lambda) = (x^2 + y^2) f(x,y; \lambda)$ where $f(x,y; \lambda) = x^2 + y^2 - \lambda$, that is, $\mathcal{X}/v$ is Hamiltonian in $\mathbb{R}^2 \backslash \{v^{-1}(0) \}$. When $\lambda=0$ it has a unstable degenerate (orbitally linearizable) focus at the origin. For $\lambda > 0$, the invariant circle $f(x,y; \lambda) = 0$ is a hyperbolic limit cycle of the system which bifurcates from the origin. In fact this is the unique limit cycle that family \eqref{Ejemplo-V-foco-ciclo} has in all $\mathbb{R}^2$ for all $\lambda \in \mathbb{R}$. When $\lambda \neq 0$, the system has a non-degenerate focus at the origin whose stability depends on the sign of $\lambda$. The weights associated to the Newton diagram are always $W(\mathbf{N}(\mathcal{X})) = \{ (1,1)\}$ althoug $\mathbf{N}(\mathcal{X})$ is not invariant in the family, it is formed by the edge with endpoints $(2,0)$ and $(0, 2)$ when $\lambda  > 0$ whereas the edge changes having now endpoints $(4, 0)$ and $(0, 4)$ when $\lambda=0$.

\subsubsection{Analysis when $\lambda > 0$:} In this case the polar system is $\dot{\rho} = \rho (-\lambda + \rho^2)$, $\dot{\varphi} = \lambda + \rho^2$ and its inverse integrating factor
$$
V(\varphi, \rho) = \frac{\rho (-\lambda + \rho^2)}{\lambda + \rho^2} =  -\rho + \frac{2}{\lambda} \rho^3 + O(\rho^4),
$$
which is of Laurent type with leading exponent $m=1$. The solution $\Pi^*(\rho)$ of the Cauchy problem obtained with the fundamental differential equation \eqref{fundamental} and the initial condition $\Pi(0)=0$ is not unique. For instance $\Pi^*(\rho) = 0$ and $\Pi^*(\rho) = \rho$ are two of such solutions. But, separating variables, we can integrate \eqref{fundamental} under the additional assumptions that the solution has a power expansion at $\rho=0$ with linear leading term and gives the closed-form expression of the analytic Poincar\'e map
$$
\Pi(\rho) = \frac{\rho^2 - \lambda + \sqrt{(\rho^2 - \lambda)^2 +
  4 \rho^2 \lambda \eta_1^2}}{2 \rho \eta_1} = \eta_1 \rho + \frac{\eta_1(1-\eta_1^2))}{\lambda} \rho^3 + O(\rho^5)
$$
having the structure predicted by Theorem \ref{lemma-fundamental}(iii). Notice that the only positive real fixed point $\rho^*$ of $\Pi$ is $\rho^* = \sqrt{\lambda}$.

We can easily compute the value of $\eta_1$ using Bautin method because $\Omega_{11} = \emptyset$. So inserting (with unknown coefficient functions of $\varphi$) an analytic flow $\Phi(\varphi; \rho_0)$ of $d \rho/ d \varphi = \rho (-\lambda + \rho^2) / (\lambda + \rho^2)$ with respect to the initial condition $\rho_0$ near $\rho_0=0$ and using that $\Pi(\rho_0) = \Phi(2 \pi; \rho_0)$ we get after some straightforward computations that $\eta_1 = \exp(-2 \pi) < 1$.

\subsubsection{Analysis when $\lambda = 0$:} Now the polar system is $\dot{\rho} = \rho$, $\dot{\varphi} = 1$ and its inverse integrating factor
$V(\varphi, \rho) = \rho$. From the general solution of the fundamental differential equations \eqref{fundamental} we obtain that the Poincar\'e map is the linear map $\Pi(\rho) = \eta_1 \rho$. In this case we also are able to compute $\eta_1$ because we easily obtain the flow $\Phi$ of the linear system $d \rho/ d \varphi = \rho$. We get $\Phi(\varphi; \rho_0) = \rho_0 \exp(\varphi)$ so that $\Pi(\rho_0) = \Phi(2 \pi; \rho_0) =  \exp(2 \pi) \rho$, hence $\eta_1 = \exp(2 \pi) > 1$.

\subsubsection{Consequences:} The monodromic parameter space is $\Lambda = \mathbb{R}$ and the origin of \eqref{Ejemplo-V-foco-ciclo} is always a focus that changes its stability at $\lambda = 0$ bifurcating one small-amplitude limit cycle. The vector field $\mathcal{X}$ and the associated differential equation \eqref{eq3} depend analytically with respect to the parameter $\lambda$ but the coefficients of the formal expansion of the Poincar\'e map do not, for example
$$
\eta_1(\lambda) = \left\{ \begin{array}{ccc}
                                \exp(-2 \pi) & \mbox{if} & \lambda > 0, \\
                                \exp(2 \pi) & \mbox{if} & \lambda = 0.
                              \end{array} \right.
$$

\subsection{Bautin's method with characteristic directions does not work even having a formal Poincar\'e map}

Let $\Phi(\varphi; \rho_0)$ be the flow associated to some weights $(p,q) \in W(\mathbf{N}(\mathcal{X}))$. It is important to remark that although $\Pi(\rho_0) = \Phi(2 \pi; \rho_0)$ has a formal expansion at $\rho_0=0$ under the assumptions of Theorem \ref{lemma-fundamental}, it is still possible that $\Phi(\varphi; \rho_0)$ may not have it for all $\varphi \in \mathbb{S}^1$ and consequently Bautin method cannot be used to compute $\eta_m$.

Indeed in general, as it is explained in the last section of \cite{GaGi3}, when $\Omega_{pq} \neq \emptyset$ the flow $\Phi(\varphi; \rho_0)$ cannot be expressed even like $\Phi(\varphi; \rho_0) = a(\varphi) \rho_0 + o(\rho_0)$ with $a(0)=1$ and $a(2 \pi) > 0$. In fact if $\Phi$ had a linear dominant term as in the previous expression then it must happen that $\Pi(\rho_0) = \eta_1 \rho_0 + o(\rho_0)$ with
$\eta_1 = a(2 \pi) = \exp(\xi_{pq})$, provided $\xi_{pq}$ defined in \eqref{xi-pq} exists, see \cite{GaGi3} again for a proof, and in particular $\xi_{pq}=0$ would be a necessary center condition.

To illustrate this behaviour we may take, as example, the family presented in Example 2 of \cite{GaGi2} given by
\begin{eqnarray}
\dot{x} &=& \lambda_1 (x^6 + 3 y^2) (-y + \mu x) + \lambda_2 (x^2+y^2)(y + A x^3), \nonumber \\
\dot{y} &=& \lambda_1 (x^6 + 3 y^2) (x + \mu y) + \lambda_2 (x^2+y^2)(-x^5 + 3 A x^2 y ), \label{ejemplo3-DCCD}
\end{eqnarray}
with parameter space $(\lambda_1, \lambda_2, \mu, A) \in \mathbb{R}^4$. The full family shares the inverse integrating factor $v(x,y) = (x^2+y^2)(x^6 + 3 y^2)$ and the origin is a monodromic singular point if and only if the parameters belong to the subset
\begin{equation}\label{monodromy2-Ejemplo3}
\Lambda = \{ (\lambda_1, \lambda_2, \mu, A) \in \mathbb{R}^4 : 3 \lambda_1 - \lambda_2 > 0, \ \ \lambda_1 - \lambda_2 > 0\}.
\end{equation}
Family \eqref{ejemplo3-DCCD} restricted to $\Lambda$ satisfies that $W(\mathbf{N}(\mathcal{X})) = \{ (1,1), (1,3) \}$. Let's choose the weights $(p, q) = (1,1)$, hence taking polar coordinates we see that $\Omega_{11} \neq \emptyset$ because the polar system becomes $\dot\rho = F_2(\varphi) \rho + O(\rho^2)$, $\dot\varphi = G_2(\varphi) +  O(\rho)$, where $F_2(\varphi) = \sin\varphi (\lambda_2 \cos\varphi +3 \mu \lambda_1 \sin\varphi)$ and $G_2(\varphi) = (3 \lambda_1-\lambda_2) \sin^2\varphi$. Then $V(\varphi, \rho) = v(\rho \cos\varphi, \rho \sin\varphi)/(\rho^3 \dot\varphi(\varphi, \rho))$ becomes a Puiseux inverse integrating factor with $m=1$. We can compute $\xi_{11}$ as follows
$$
\xi_{11} = PV \int_0^{2 \pi} \frac{F_2(\varphi)}{G_2(\varphi)} d \varphi = PV \int_0^{2 \pi} \frac{3 \lambda_1 \mu  + \lambda_2 \cot\varphi}{3 \lambda_1 - \lambda_2} d \varphi = \frac{6 \pi \lambda_1 \mu}{3 \lambda_1 - \lambda_2},
$$
and we check that the origin can be a center of \eqref{ejemplo3-DCCD} with $\xi_{11} \neq 0$ because in \cite{GaGi3} it is proved that $3 \lambda_1 \mu +  \sqrt{3} A \lambda_2 = 0$ is the unique center condition under some additional parameter restrictions which do not imply $A \lambda_2 = 0$. In summary $\eta_1 \neq \exp(\xi_{11})$ and similar computations with the other weights $(p,q)=(1,3)$ reveal that $\eta_1 \neq \exp(\xi_{13})$ as well.

\section{Proofs}

\subsection{Proof of Lemma \ref{V->puiseux-laurent}}

\begin{proof}
On one hand the Puiseux series $V(\varphi, \rho) = \sum_{i \geq m} v_i(\varphi) \rho^{i/n}$ is transformed into a power series $V(\varphi, \sigma^n) = \sum_{i \geq m} v_i(\varphi) \sigma^{i}$. On the other hand, the Jacobian $D \phi(\sigma)$ of the change $\rho = \phi(\sigma) = \sigma^n$ is $D \phi(\sigma) = n \sigma^{n-1}$. Taking into account that inverse integrating factors are transformed under changes of variables as $\tilde{V}(\varphi, \sigma) = V(\varphi, \phi(\sigma))/D \phi(\sigma)$, the lemma follows.
\end{proof}

\subsection{Proof of Theorem \ref{teo-Fund-VPi}}

The proof of Theorem \ref{teo-Fund-VPi} follows joining Propositions \ref{Teorema-main-V-Pi} and \ref{prop-V-Piprime}. Before we need a preliminary result.

We define the set
\begin{equation}\label{C*}
C^* := C \backslash \{ \Theta^{-1}(0) \cup\{\rho=0\} \},
\end{equation}
and we analyze its connected components.

\begin{proposition}\label{prop-v-cr}
Let $U$ be a connected component of $C^*$. Then the vector field $\hat{\mathcal{X}} = \partial_\varphi + \mathcal{F}(\varphi, \rho) \partial_\rho$ is $C^k(U)$ and possesses a $C^k$ inverse integrating factor in $U$ with arbitrary positive integer $k$.
\end{proposition}
\begin{proof}
Clearly $\mathcal{F} \in C^k(U)$ since $\Theta^{-1}(0) \not\subset U$. Moreover, $U$ is a {\it canonical region} for the flow of $\hat{\mathcal{X}}$, that is, $U$ is an open connected component in the complement in $C$ of the union of the separatrices of $\hat{\mathcal{X}}$. Then the flow $\Phi|_{U}$ of $\hat{\mathcal{X}}$ restricted to $U$ is $C^k$-equivalent to the annular flow $\psi$ generated by the vector field $\mathcal{Z} = \partial_\varphi$ in $C$, see a proof in \cite{LLNZ}. That means that there exists a $C^k$-diffeomorphism $\xi$ from $U$ onto $C$ which takes orbits of $\Phi$ onto orbits of $\psi$ preserving or reversing simultaneously the sense of all orbits. Since $\hat{\mathcal{X}}= \xi^* \mathcal{Z}$ and $V(\varphi, \rho) = f(\rho)$ is a $C^k$ inverse integrating factor of $\mathcal{Z}$ for any real $C^k$ function $f$, it follows that $\hat{V} = J_\xi^{-1} V \circ \xi$ is a $C^k$ inverse integrating factor of $\hat{\mathcal{X}}$ in $U$, where $J_\xi$ denotes the Jacobian determinant of $\xi$.
\end{proof}

Let $\Phi(\varphi; \rho_0)$ be the flow of the vector field $\hat{\mathcal{X}} = \partial_\varphi + \mathcal{F}(\varphi, \rho) \partial_\rho$ with initial condition $\Phi(0; \rho_0) = \rho_0 >0$ sufficiently small. We define the quantity
\begin{equation}\label{I-def}
\mathcal{I}(\rho_0) = PV \int_0^{2 \pi} \frac{\partial \mathcal{F}}{\partial \rho}(\varphi, \Phi(\varphi; \rho_0)) \, d \varphi,
\end{equation}
provided it exists, and it will play a fundamental role in the next results.

Since $C^*$ is the union of canonical regions in the cylinder $C$ and $\hat{\mathcal{X}} \in C^1(C^*)$, there is always an inverse integrating factor $V \in C^1(C^*)$ by Proposition \ref{prop-v-cr}. If $V$ satisfies additionally that $V \in C(\Theta^{-1}(0))$ then we have the following result.

\begin{proposition}\label{Teorema-main-V-Pi}
Assume the existence of an inverse integrating factor $V \in C^1(C^*)$ such that $V \in C(\Theta^{-1}(0))$ holds. If $V(0, \rho_0) \neq 0$ then $\mathcal{I}(\rho_0)$ exists and
\begin{equation}\label{Pi-eq}
V(0, \Pi(\rho_0)) = V(0, \rho_0) \, \exp\left( \mathcal{I}(\rho_0) \right).
\end{equation}
\end{proposition}
\begin{proof}
The partial differential equation that satisfies $V$ is
\begin{equation}\label{EDP-V_vector}
\hat{\mathcal{X}}(V) = V \, \partial_\rho \mathcal{F}.
\end{equation}
In the simplest case that $\Omega_{pq} = \emptyset$ it follows that $\Theta^{-1}(0) \backslash \{\rho=0\} = \emptyset$ and therefore $\mathcal{F} \in C^1(C)$ and $V \in C^1(C \backslash \{\rho=0\})$. Since $\{ \rho= 0 \} \cap \Phi(\varphi; \rho_0) = \emptyset$, evaluating \eqref{EDP-V_vector} along the smooth flow $\Phi(\varphi; \rho_0)$ gives
\begin{equation}\label{evaluar-flow}
\frac{d}{d \varphi} V(\varphi, \Phi(\varphi; \rho_0)) = \frac{\partial \mathcal{F}}{\partial \rho}(\varphi, \Phi(\varphi; \rho_0)) V(\varphi, \Phi(\varphi; \rho_0)).
\end{equation}
Integrating this equation we easily obtain the expression of $V$ on the flow:
\begin{equation}\label{V-on-flow-0}
V(\varphi, \Phi(\varphi; \rho_0)) = V(0, \rho_0) \, \exp\left( \int_0^\varphi \frac{\partial \mathcal{F}}{\partial \rho}(\sigma, \Phi(\sigma; \rho_0)) \, d \sigma \right).
\end{equation}
In particular, we see that the zero-set $V^{-1}(0) = \{ (\varphi, \rho) \in C : V(\varphi, \rho) =  0 \}$ is invariant for the flow of $\hat{\mathcal{X}}$. In other words, if $V(0, \rho_0) = 0$ for some $\rho_0 > 0$ then $V(\varphi, \Phi(\varphi; \rho_0)) \equiv 0$ for all $\varphi \in \mathbb{S}^1$.
\newline

Now we focus in the much more involved case with $\Theta^{-1}(0) \backslash \{\rho=0\} \neq \emptyset$, hence in particular $\Omega_{pq} \neq \emptyset$. Then $\mathcal{F}$ is no longer $C^1(C)$, rather $\mathcal{F} \in C^1(C \backslash \Theta^{-1}(0))$.

Let $U_i$, with $i=0, 1,\ldots, \kappa$, be the connected components of $C^*$. By the previous analysis and since both $\mathcal{F}$ and $V$ are of class $C^1(U_i)$, we stress that $V$ cannot have isolated zeros in $U_i$. Even more, since additionally $V \in C(\Theta^{-1}(0))$, the whole solution curve $\gamma_{\rho_0} = \{ (\varphi, \Phi(\varphi; \rho_0)) : \varphi \in \mathbb{S}^1 \}$ is contained in $V^{-1}(0)$ if and only if  $V(0, \rho_0) = 0$.

The curve $\gamma_{\rho_0}$ intersects $\Theta^{-1}(0)$ at some points $(\bar\varphi_i(\rho_0), \Phi(\bar\varphi_i(\rho_0); \rho_0)) \in  \Theta^{-1}(0)$ with $\bar{\varphi}_i(\rho_0) \in \mathbb{S}^1$ for $i=1, \ldots, \kappa$. Without loss of generality we may assume that $0 < \bar\varphi_i(\rho_0) < 2 \pi$ so that $\gamma_{\rho_0}$ contains the arcs $\{ (\varphi, \Phi(\varphi; \rho_0)) : 0 \leq \varphi \leq \bar\varphi_{1}(\rho_0) \} \subset U_0$, $\{ (\varphi, \Phi(\varphi; \rho_0)) : \bar\varphi_i(\rho_0) \leq \varphi \leq \bar\varphi_{i+1}(\rho_0) \} \subset U_i$ for $i=1, \ldots, \kappa-1$, and $\{ (\varphi, \Phi(\varphi; \rho_0)) : \bar\varphi_{\kappa}(\rho_0) \leq \varphi \leq 2 \pi \} \subset U_\kappa$.

We will evaluate \eqref{EDP-V_vector} along the flow $\Phi(\varphi; \rho_0)$ with initial condition $\rho_0$ such that $V(0, \rho_0) \neq 0$, hence $\gamma_{\rho_0} \cap V^{-1}(0) = \emptyset$, and with angles restricted to $\varphi \in \bar{I}_\varepsilon(\rho_0) = [0, 2 \pi] \backslash \bar{J}_\varepsilon(\rho_0)$ with $\bar{J}_\varepsilon(\rho_0) = \cup_{i=1}^\kappa (\bar\varphi_i(\rho_0)-\varepsilon, \bar\varphi_i(\rho_0)+\varepsilon)$. This gives equation \eqref{evaluar-flow} but only for the angles $\varphi \in \bar{I}_\varepsilon(\rho_0)$. Integrating \eqref{evaluar-flow} over $[0, \varphi] \cap \bar{I}_\varepsilon(\rho_0)$ and taking the limit when $\varepsilon \to 0^+$ we get
\begin{equation}\label{int-vF}
PV \int_{0}^{\varphi} \frac{\frac{d}{d \sigma} V(\sigma, \Phi(\sigma; \rho_0))}{V(\sigma, \Phi(\sigma; \rho_0))} \, d \sigma = PV \int_0^\varphi \frac{\partial \mathcal{F}}{\partial \rho}(\sigma, \Phi(\sigma; \rho_0)) \, d \sigma.
\end{equation}

On the other hand, the function $\bar{P}_1(\sigma; \rho_0) = \log |V(\sigma, \Phi(\sigma; \rho_0))|$ is clearly a continuous primitive in $\bar{I}_\varepsilon(\rho_0)$ of the integrand of the left hand side in \eqref{int-vF}. But since $V$ is continuous on $\Theta^{-1}(0) \backslash \{\rho=0\}$ and $\Phi(.; \rho_0)$ is continuous at $\bar{\varphi}_i(\rho_0)$ for $i=1, \ldots, \kappa$, it follows that $\bar{P}_1(\sigma; \rho_0)$ is indeed continuous in $[0, \varphi]$. Therefore we can write
\begin{eqnarray*}
PV \int_0^\varphi \frac{\partial \mathcal{F}}{\partial \rho}(\sigma, \Phi(\sigma; \rho_0)) \, d \sigma &=& \lim_{\varepsilon \to 0^+}  \int_{[0, \varphi] \cap \bar{I}_\varepsilon(\rho_0)} \frac{\frac{d}{d \sigma} V(\sigma, \Phi(\sigma; \rho_0))}{V(\sigma, \Phi(\sigma; \rho_0))} \, d \sigma  \\
 &=&  \lim_{\varepsilon \to 0^+} \sum_i \bar{P}_1(\bar\varphi_{i+1}(\rho_0) - \varepsilon; \rho_0) - \bar{P}_1(\bar\varphi_i(\rho_0) + \varepsilon; \rho_0) \\
 &=& \bar{P}_1(\varphi; \rho_0) - \bar{P}_1(0; \rho_0) = \log \left| \frac{V(\varphi, \Phi(\varphi; \rho_0))}{V(0, \rho_0)} \right|.
\end{eqnarray*}
In particular $\mathcal{I}(\rho_0)$ exists. Also from here we deduce
\begin{equation}\label{V-on-flow}
V(\varphi, \Phi(\varphi; \rho_0)) = V(0, \rho_0) \, \exp\left( PV \int_0^\varphi \frac{\partial \mathcal{F}}{\partial \rho}(\sigma, \Phi(\sigma; \rho_0)) \, d \sigma \right).
\end{equation}
Evaluating this equation at $\varphi = 2 \pi$ and using the $2 \pi$-periodicity of $V$ in the variable $\varphi$ we obtain relation \eqref{Pi-eq}.
\end{proof}

\begin{proposition} \label{prop-V-Piprime}
Restricting family \eqref{DCC-1} to $\Lambda \backslash \Lambda_{pq}$ one has $\Pi'(\rho_0) = \exp\left( \mathcal{I}(\rho_0) \right)$.
\end{proposition}
\begin{proof}
When the family is restricted to $\Lambda \backslash \Lambda_{pq}$ then $\Theta^{-1}(0) \backslash \{\rho=0\} = \emptyset$ and $C^*$ is just $C \backslash \{\rho=0\}$. In particular $C^*$ is open and connected and there is $V \in C^1(C^*)$ from Proposition \ref{prop-v-cr}. Therefore Proposition \ref{Teorema-main-V-Pi} works and the condition $V \in C(\Theta^{-1}(0))$ is not necessary.
\newline

On the other hand, by definition, $\Phi(\varphi; \rho_0)$ satisfies
$$
\frac{\partial \Phi}{\partial \varphi}(\varphi; \rho_0) = \mathcal{F}(\varphi, \Phi(\varphi; \rho_0)), \ \ \Phi(0; \rho_0) = \rho_0 > 0,
$$
for all $\varphi \in [0, 2 \pi]$. We can differentiate both expressions with respect to $\rho_0$ and permute the derivation order to get the first variational equation
$$
\frac{\partial}{\partial \varphi} \left( \frac{\partial \Phi}{\partial \rho_0}(\varphi; \rho_0) \right) = \frac{\partial \mathcal{F}}{\partial \rho}(\varphi, \Phi(\varphi; \rho_0)) \, \frac{\partial \Phi}{\partial \rho_0}(\varphi; \rho_0), \ \ \ \frac{\partial \Phi}{\partial \rho_0}(0; \rho_0) = 1.
$$
This is a scalar linear homogeneous first order differential equation with initial condition different from zero, hence its solution $\frac{\partial \Phi}{\partial \rho_0}(\varphi; \rho_0)$ never vanishes. Integrating this equality on $[0, \varphi]$ gives
\begin{eqnarray*}
\int_{0}^\varphi \frac{\partial \mathcal{F}}{\partial \rho}(\sigma, \Phi(\sigma; \rho_0)) \, d \sigma &=& \int_{0}^\varphi  \frac{\frac{\partial}{\partial \varphi} \left( \frac{\partial \Phi}{\partial \rho_0}(\sigma; \rho_0) \right)}{\frac{\partial \Phi}{\partial \rho_0}(\sigma; \rho_0)} d \sigma.
\end{eqnarray*}
The function $\bar{P}_2(\sigma; \rho_0) = \log\left| \frac{\partial \Phi}{\partial \rho_0}(\sigma; \rho_0) \right|$ is a continuous primitive in $[0, 2 \pi]$ of the integrand of the right-hand side, hence we obtain
\begin{eqnarray*}
\int_{0}^\varphi \frac{\partial \mathcal{F}}{\partial \rho}(\sigma, \Phi(\sigma; \rho_0)) \, d \sigma &=& \bar{P}_2(\varphi; \rho_0) - \bar{P}_2(0; \rho_0)  = \log\left| \frac{\partial \Phi}{\partial \rho_0}(\varphi; \rho_0) \right|.
\end{eqnarray*}
Evaluating this expression at $\varphi = 2 \pi$ yields $\mathcal{I}(\rho_0) = \log\left| \Pi'(\rho_0) \right|$ or equivalently $\Pi'(\rho_0) = \exp\left( \mathcal{I}(\rho_0) \right)$.
\end{proof}

\subsection{Proof of Theorem \ref{lemma-fundamental}}

\begin{proof}
We restrict the analysis to $m \geq 1$. Thus we consider the analytic vector field $\mathcal{Y} = V(0, \rho) \partial_\rho + V(0, \Pi) \partial_\Pi$ associated to the fundamental equation \eqref{fundamental}, that is, $V(0, \Pi(\rho)) = V(0, \rho) \, \Pi'(\rho)$. We look for the invariant branches $\Pi^*(\rho)$ of this differential equation at the origin, that is, we search for asymptotic expansion of solutions  $\Pi^*(\rho)$ of this equation with initial condition $\Pi^*(0) = 0$. We remark that the former Cauchy problem does not has uniqueness of solutions in general, that the trivial solution $\Pi^*(\rho) = \rho$ is always present, and that the Poincar\'e map $\Pi(\rho)$ is contained in the set of such invariant branches.

First of all we claim that the Dulac asymptotic expansion \eqref{Poinc-expan} of $\Pi$ has no logarithmic terms, that is, the polynomials $P_j$ in \eqref{Poinc-expan} are just constants and the asymptotic expansion \eqref{Poinc-expan} is indeed a Puiseux series. The proof of this claim comes after a tedious but straight computation just inserting \eqref{Poinc-expan} into the differential equation \eqref{fundamental} and equating the coefficients of the different terms $\rho^\alpha \log^j(\rho)$ with $\alpha \in \mathbb{R}$ and $j \in \mathbb{N}$. This nice behaviour of $\Pi$ is a consequence of the fact that $\Pi$ must satisfy the differential equation \eqref{fundamental}.

The Newton diagram $\mathbf{N}(\mathcal{Y})$ has one segment with endpoints $(1,m)$ and $(m,1)$, hence with weight $(p,q)=(1,1)$. This confirms the well-known fact that the Poincar\'e map $\Pi$, as well as all the invariant branches at the origin of \eqref{fundamental},  has an expansion of the form $\Pi(\rho) = \eta_1 \rho + o(\rho)$.

Since $0 \not\in \Omega_{pq}$ it follows that $v_m(0) \neq 0$ and without lost of generality we may take $v_m(0) = 1$ just dividing $V$ by $v_m(0)$.

The index $n$ of the asymptotic expansion of a branch $\Pi^*(\rho) = \eta_1 \rho + o(\rho)$ will be computed using Remark \ref{Demina-remark}.  The reduced equation of the dominant balance $E_0(\Pi(\rho), \rho) = \rho^m \Pi'(\rho) - \Pi^m$ and its formal G\^{a}teaux derivative at $\eta_1 \rho$ which is
\begin{eqnarray*}
\frac{\delta E_0}{\delta \Pi}[\eta_1 \rho] &=& \lim_{s \to 0} \frac{E_0[\eta_1 \rho + s \rho^{1+j}, \rho] - E_0[\eta_1 \rho, \rho]}{s} = \Xi(j) \rho^{m+j}
\end{eqnarray*}
with $\Xi(j) = 1 + j - m \not\equiv 0$. Then the Fuchs index is the root of $\Xi$, that is, $j = m-1 \not\in \mathbb{Q}^+ \backslash \mathbb{N}$ so that $n=p = 1$ and $\Pi^*$ has an expansion of the form $\Pi^*(\rho) = \sum_{j \geq 1} \eta_j \rho^j$. This proves that the Poincar\'e map $\Pi$ has a formal power series expansion. We notice that we cannot guarantee this formal series $\Pi^*(\rho)$ possesses uniquely determined coefficients since the Fuchs index $m-1 \in \mathbb{Q}^+ \cup \{0\}$.

The polynomial $\mathcal{Q}$ defined in \eqref{det-pol-2} is given in our analysis by the equation
$$
\rho^m \frac{d}{d \rho}(\eta \rho) - (\eta \Pi)^m = \mathcal{Q}(\eta) \rho^{r+1},
$$
that is, $\mathcal{Q}(\eta) = \eta (1-\eta^{m-1})$. We continue assuming $m > 1$ so that $\mathcal{Q}(\eta) \not\equiv 0$ and $\eta_1$ must is a real root of $\mathcal{Q}$. The non-zero roots of $\mathcal{Q}$ are $\eta = \zeta_k = \exp(2 k \pi i/ (m-1))$ with $k=0, \ldots, m-2$, all roots of unity. Therefore $\eta_1 = \zeta_0 = 1$ since $\eta_1$ must be a positive real root of $\mathcal{Q}$. We emphasize that although $1$ is a simple root of $\mathcal{Q}$ because $\mathcal{Q}'(1) = 1-m \neq 0$ this does not imply the uniqueness of the branches $\Pi^*(\rho) = \rho + \cdots$, hence we still do not know if the origin is a center or not.

Now we insert the expression of $\Pi$ and that of $V$ given in \eqref{V-Cw-def} into equation \eqref{fundamental} and we get
\begin{equation}\label{power-series-fund}
\sum_{i \geq m} v_i(0) \rho^i \left(1 + \sum_{j \geq 2} \eta_j \rho^{j-1} \right)^{i} = \left( \sum_{i \geq m} v_i(0) \rho^{i} \right) \left( 1 + \sum_{j \geq 2} j \eta_j \rho^{j-1} \right).
\end{equation}
We equate the coefficients of like powers of $\rho$ in this equation. For example, from the power $\rho^{m+1}$ we get $(m-2) \eta_2 = 0$. After the forthcoming detailed analysis we claim that $\eta_i = 0$ for $2 \leq i \leq m-1$ and that when $\eta_m = 0$ then $\eta_i = 0$ for all $i > m$ and this will prove statement (ii). We will prove the claim using the following formulas for power series raised to integer powers and product of two power series.

\begin{remark}\label{series-product}
{\rm Given two power series $\sum_{i \geq 0} a_i x^i$ and $\sum_{i \geq 0} b_i x^i$ and $n \in \mathbb{N}$ we have $\left( \sum_{i \geq 0} a_i x^i \right) \left( \sum_{i \geq 0} b_i x^i \right) = \sum_{i \geq 0} c_i x^i$ with $c_i = \sum_{j+k=i} a_j b_k$, and $\left( \sum_{i \geq 0} a_i x^i \right)^n = \left( \sum_{i \geq 0} d_i x^i \right)$ where the coefficients $d_i$ can be calculated recursively by $d_0 = a_0^n$ and $d_i = (i a_0)^{-1}  \sum_{k=1}^i (k n -i + k) a_k d_{i-k}$ for $i \geq 1$, provided that $a_0 \neq 0$. }
\end{remark}

In our case we get
$$
\left(1 + \sum_{j \geq 2} \eta_j \rho^{j-1} \right)^{i} = \sum_{j \geq 0} \bar{\eta}_j^{(i)} \rho^j
$$
with $\bar{\eta}_0^{(i)} = 1$ and $\bar{\eta}_j^{(i)} = j^{-1} \sum_{k=1}^j (k i-j+k) \eta_{k+1} \bar{\eta}_{j-k}^{(i)}$ for $j \geq 1$. Therefore the left-hand side of \eqref{power-series-fund} is
$$
\sum_{i \geq m} v_i(0) \rho^i \left(1 + \sum_{j \geq 2} \eta_j \rho^{j-1} \right)^{i} = \sum_{i \geq m} v_i(0) \rho^i \left(\sum_{j \geq 0} \bar{\eta}_j^{(i)} \rho^j\right) = \sum_{i \geq m} \sum_{j \geq 0} v_i(0) \bar{\eta}_j^{(i)} \rho^{i+j}
$$
that we rewrite as
$$
\sum_{k \geq m} \alpha_k \rho^k, \ \ \mbox{with} \ \alpha_k = \sum_{i+j = k} v_i(0) \bar{\eta}_j^{(i)}.
$$
On the other hand, the right-hand side of \eqref{power-series-fund} is
$$
\left( \sum_{i \geq m} v_i(0) \rho^{i} \right) \left( 1 + \sum_{j \geq 2} j \eta_j \rho^{j-1} \right) = \left( \sum_{i \geq m} v_i(0) \rho^{i} \right) \left(\sum_{j \geq 1} j \eta_j \rho^{j-1} \right)  = \sum_{k \geq 0} c_k \rho^k,
$$
where $c_k = \sum_{i+j=k} v_j(0) (i+1) \eta_{i+1}$ with the convention $v_j(0) = 0$ when $j < m$. In particular, $c_k=0$ for $k < m$.

In short, \eqref{power-series-fund} holds if and only if $c_k = \alpha_k$ for $k \geq m$. The equality $c_m = \alpha_m$ is just an identity taking into account that $v_m(0) = \eta_1 = 1$. In summary, the following relations hold:
\begin{equation}\label{iteration-eta}
\sum_{i+j=k} v_j(0) (i+1) \eta_{i+1} = \sum_{i+j = k} v_i(0) \bar{\eta}_j^{(i)}, \ \ k \geq m+1,
\end{equation}
convention $v_j(0) = 0$ when $j < m$, $v_m(0) = 1$ and $\eta_1 = 1$ and the definitions $\bar{\eta}_0^{(i)} = 1$ and $\bar{\eta}_j^{(i)} = j^{-1} \sum_{s=1}^j (s i-j+s) \eta_{s+1} \bar{\eta}_{j-s}^{(i)}$ for $j \geq 1$.

Unfolding the relations \eqref{iteration-eta} for each $k \geq m+1$ gives a sequence of equations of the form
\begin{eqnarray} \label{system-eta-v}
0 &=& (m-2) \eta_2, \nonumber \\
0 &=& (m-3) \eta_3 + (m-1) \eta_2 \left( v_{m+1}(0) + \frac{1}{2} m \eta_2 \right), \nonumber \\
  &\vdots &  \\
0 &=& (m-j) \eta_j + (m-j+2) \eta_{j-1} (*) + (m-j+4) \eta_{j-2} (*) + \cdots \nonumber\\
 & &  + (m-1) \eta_{2} (*), \nonumber
\end{eqnarray}
for $j \geq 2$ where the terms $(*)$ in the expression of $\eta_{\ell} (*)$ for $j-1 \leq \ell \leq 2$ are polynomials in the variables $v_{i}(0)$ with $m+1 \leq i \leq j$ and $\eta_s$ with $2 \leq s \leq \ell$.

From the first equation \eqref{system-eta-v} we get that $\eta_2 = 0$ when $m>2$. From the first and the second equations \eqref{system-eta-v} one see that $\eta_2 = \eta_3 = 0$ when $m>3$. Continuing this kind of arguments, we conclude the first part of the former claim, that is, $\eta_i = 0$ for $2 \leq i \leq m-1$. The second part of the claim follows taking into account the first part together with the last equation \eqref{system-eta-v}. More precisely, when we take $j=m$ in the last equation \eqref{system-eta-v} it reduces to the identity $0=0$ and $\eta_m$ remains arbitrary. Now we move forward to values $j > m$ for which the last equation \eqref{system-eta-v} becomes:
\begin{itemize}
  \item When $j=m+1$ then $0 = - \eta_{m+1} + \eta_{m} (*)$;
  \item When $j=m+2$ then $0 = - 2 \eta_{m+2} + \eta_{m} (**)$;
\end{itemize}
and so on. In summary, taking $j=m + \ell$ with $\ell \geq 1$ the last equation \eqref{system-eta-v} becomes
\begin{equation}\label{relation_eta_m}
0 = - \ell \eta_{m+\ell} + \eta_{m} (**)
\end{equation}
with $(**)$  certain polynomial in the variables $v_{i}(0)$ with $m+1 \leq i \leq m+\ell$ and $\eta_m$. Equation \eqref{relation_eta_m} shows that when $\eta_m = 0$ then $\eta_i = 0$ for all $i > m$ and the statement (ii) of the theorem is proved.
\newline

The case $m=1$ corresponding to the degeneracy $\mathcal{Q}(\eta) \equiv 0$ can be analyzed in a similar way. We insert the expansion $\Pi(\rho) = \sum_{j \geq 1} \eta_j \rho^j$ and that of $V$ given in \eqref{V-Cw-def} with $m=1$ into equation \eqref{fundamental} and, instead of \eqref{power-series-fund}, we get
\begin{equation}\label{power-series-fund-2}
\sum_{i \geq 1} v_i(0) \rho^i \left(\sum_{j \geq 1} \eta_j \rho^{j-1} \right)^{i} = \left( \sum_{i \geq 1} v_i(0) \rho^{i} \right) \left( \sum_{j \geq 1} j \eta_j \rho^{j-1} \right).
\end{equation}
We use the formulas of Remark \ref{series-product} to rearrange the series in \eqref{power-series-fund-2}. After a careful analysis we can prove that the possible branches have the form $\Pi^*(\rho) = \eta_1 \rho + \cdots$ and once we fix $\eta_1$ then the branch is unique. Thus if $\eta_1 = 1$ then the origin is a center and statement (iii) is proved.
\end{proof}

\subsection{Proof of Lemma \ref{structure-main-2-new}}

\begin{proof}
Any Puiseux inverse integrating factor $V$ is defined in almost everywhere point of a small enough punctured neighborhood of $\{ \rho = 0 \}$ in the cylinder \eqref{cylinder}. More specifically $V$ is defined in $\hat{C} = C \backslash \{ \{ \rho = 0 \} \cup \Theta^{-1}(0) \cup \Delta^* \}$ where we define $\Delta^*$ as the set of lines $\Delta^* = \{ (\theta^*, \rho) \in C : \theta^* \in \Omega_{pq} \}$.

We recall that $\{\rho=0\}$ is a polycycle of \eqref{eq3} since the origin is a monodromic singularity of system \eqref{DCC-1}. Moreover, $V^{-1}(0) \cap \hat{C}$ is an invariant set of \eqref{eq3}. Therefore, taking the cylinder $C$ sufficiently small, if $V^{-1}(0) \backslash \{ \rho=0\} \neq \emptyset$ then either $V^{-1}(0)$ contains a sequence of periodic orbits of \eqref{eq3} accumulating at $\{ \rho=0\}$ (center case) or $V^{-1}(0)$ has $\{\rho=0\}$ as limit set (focus case). We are going to see that both cases are impossible. To see it we will use Lemma \ref{V->puiseux-laurent} by doing the change $\sigma^n = \rho$ with $n$ the index of $V$. This change keeps invariant the monodromy and transforms $V$ into a Laurent inverse integrating factor $\tilde V(\varphi, \rho) = \sum_{i \geq m} v_i(\varphi) \rho^{i}$. We take an angle $\hat\varphi \not\in \Omega_{pq}$ such that the line $L = \{ (\hat\varphi, \rho) \in  \hat{C} \cup \{\rho=0 \} \}$ exists and next we consider the function $f(\rho) = \rho^k \tilde V(\hat\varphi, \rho)$ with $k > m$ defined in $L$. Clearly $f$ is an analytic function at $\rho=0$ and consequently cannot possess a sequence of zeros accumulating at $\rho=0$. This fact proves the lemma.
\end{proof}

\subsection{Proof of Theorem \ref{lemma-fundamental-2}}

\begin{proof}
The key point of this proof comes from a result of \cite{Ga-Gi} where it is proved that $G$ can be rewritten as
\begin{equation}\label{def-G}
G(r) =  \int_{r}^{\Pi(r)} \frac{d \rho}{V(0, \rho)}.
\end{equation}

In the first part of the proof we are going to compute $\lim_{r \to 0^+} G(r)$ in each case $m=1$ and $m > 1$ separately.

We take $m=1$ and we expand \eqref{def-G} as follows
$$
G(r) =  \int_{r}^{\Pi(r)} \frac{d \rho}{\rho + \sum_{k \geq 2} v_k(0) \rho^k} =  \int_{r}^{\Pi(r)} \frac{1}{\rho} + R_1(\rho)   d \rho,
$$
where $R_1$ is an analytic function at $\rho=0$. Recall that $v_k(0) \in \mathbb{R}$ since $0 \not\in \Omega_{pq}$. Therefore
$$
G(r) =  \log\left( \frac{\Pi(r)}{r}\right) +  \int_{r}^{\Pi(r)} R_1(\rho)   d \rho = \log\left( \eta_1 + O(r) \right) +  \int_{r}^{\Pi(r)} R_1(\rho)   d \rho,
$$
where in the last step we use that $\Pi(r) = \eta_1 r + O(r^2)$. Thus $\lim_{r \to 0^+} G(r) = \log\left( \eta_1 \right)$.
\newline

Now we take $m > 1$ and we can assume without lost of generality that $v_m(0) = 1$ because, by statement (i) of Lemma 16 in \cite{GaGi2} the eventual zeros of $v_m$ must lie in $\Omega_{pq}$ but $0 \not\in \Omega_{pq}$. Then \eqref{def-G} is written as
$$
G(r) =  \int_{r}^{\Pi(r)} \frac{d \rho}{\rho^m + \sum_{k > m} v_k(0) \rho^k} = \int_{r}^{\Pi(r)} \left( \sum_{m \leq i \leq 1} \frac{\beta_i}{\rho^i} \right) + R_m(\rho)   d \rho,
$$
where $\beta_m = 1$ and $\beta_i \in \mathbb{R}$ for $i < m$ and $R_m$ is an analytic function at $\rho=0$. Integration leads to
\begin{eqnarray*}
G(r) &=&  \left[ \sum_{i=2}^m \frac{\beta_i}{(1-i) \rho^{i-1}} \right]_{r}^{\Pi(r)} + \beta_1 \log\left( \frac{\Pi(r)}{r}\right) + \int_{r}^{\Pi(r)} R_m(\rho) d \rho \\
 &=& \sum_{i=2}^m \frac{\beta_i}{1-i} \Delta_{i-1}(r) + \beta_1 \log\left( \frac{\Pi(r)}{r}\right) + \int_{r}^{\Pi(r)} R_m(\rho) d \rho,
\end{eqnarray*}
where, for any $k \in \mathbb{N}$, we have defined
$$
\Delta_k(r) := \frac{1}{\Pi^k(r)} - \frac{1}{r^k}.
$$
Taking into account that $\Pi(r) = r + \eta_m r^m (1 + O(r))$ one has $\Pi^k(r) = r^k (1 + O(r))$ and
$$
\Delta_k(r) = \frac{r^k - \Pi^k(r)}{r^k \Pi^k(r)}  = \frac{- \eta_m r^{m+k-1}(k + O(r))}{r^{2k} (1 + O(r))} = - \eta_m r^{m-k-1} (k + O(r)).
$$
Then $G$ is written as
\begin{eqnarray*}
G(r) &=& - \eta_m  \sum_{i=2}^m \frac{\beta_i}{1-i} r^{m-i} (i-1 + O(r)) + \\
 & &  \beta_1 \log\left( 1 + \eta_m r^{m-1} (1 + O(r)) \right) + \int_{r}^{\Pi(r)} R_m(\rho) d \rho.
\end{eqnarray*}
Therefore we conclude that $\lim_{r \to 0^+} G(r) = \eta_m$.
\newline

In the second part of the proof we are going to compute de first derivative $G'(r)$ by using Leibnitz rule in \eqref{def-G} and taking into account the fundamental relation \eqref{fundamental}. Recall that we can use Leibnitz rule because in the definition domain $I^+$ of $G$ the functions $\Pi(r)$ and $V(0,r)$ are analytic since $0 \not\in \Omega_{pq}$ and $V(0,r) \neq 0$ from statement (iii) of Lemma 18 of \cite{GaGi2}. Additionally \eqref{fundamental} works $\Lambda \backslash \Lambda_{pq}$ by Theorem \ref{teo-Fund-VPi}. The outcome is that
$$
G'(r) = \frac{1}{V(0, \Pi(r))} \Pi'(r) - \frac{1}{V(0, r)} \equiv 0,
$$
so $G(r)$ is a constant, called $\mathfrak{g}$ is the statement of the theorem, and the proof finishes using that $\lim_{r \to 0^+} G(r) = \mathfrak{g}$.
\end{proof}

\subsection{Proof of Theorem \ref{cor-fundamental}}

\begin{proof}
We restrict $\mathcal{X}$ to $E \subset \Lambda$ since then it makes sense the Bautin ideal $\mathcal{B}$. The fact that $\mathcal{B}$ is principal is just a direct consequence of Theorem \ref{lemma-fundamental}.

To prove statement (i) we define $f(\rho) = V(0, \rho) = \rho^m + \cdots$ so that the differential equation \eqref{fundamental} is written as $f(\Pi) = f(\rho) \, \Pi'(\rho)$. Separating variables one has that this equation has a first integral $H(\rho, \Pi) = F(\rho) - F(\Pi)$ where $F'(\rho) = 1/f(\rho)$. Of course $H$ is not well defined in a neighborhood of $(\rho, \Pi) = (0,0)$ since $1/f(\rho)$ has a pole of order $m$ at $\rho=0$.
\newline

Integrating term-by-term the Laurent series $1/f(\rho)$, that is convergent in a punctured neighborhood of $\rho=0$, we deduce that $F(\rho)$ has a pole of order $m-1 > 1$ at $\rho=0$ since by hypothesis the residue ${\rm Res}(1/f(\rho), 0) = 0$. More specifically, $F(\rho) = \frac{c}{\rho^{m-1}} + \cdots$ with $c = 1/(1-m) \neq 0$. In particular, for any $C_m \in \mathbb{C}$, the function $\hat{F}(\rho, \Pi) =  (\rho \Pi)^{m-1} (H(\rho, \Pi) - C_m)$ is analytic at $(\rho, \Pi) = (0,0)$ and its local zero-set contains, besides the artificial lines $\rho \Pi = 0$, the level curves of $H$ which are those in which we are really interested because the Poincar\'e map $\Pi(\rho)$ must satisfy $H(\rho, \Pi(\rho)) = C_m^*(\eta_m, \lambda) \in \mathbb{C}$ for any $\rho \geq 0$ sufficiently small and some concrete $C_m^*$.

The Taylor expansion of $\hat{F}$ at the origin is $\hat{F}(\rho, \Pi) = c (\Pi^{m-1} - \rho^{m-1}) + \cdots$ where the dots are higher order terms. When $m=2$ the Implicit Function Theorem applies to $\hat{F}(\rho, \Pi) = 0$ at $(\rho, \Pi) = (0,0)$, hence the Poincar\'e map $\Pi(\rho)$ is analytic at $\rho=0$.

The analytic curve $\hat{F}(\rho, \Pi) = 0$ has a critical point at $(\rho, \Pi) = (0,0)$ in case that $m > 2$. Here we also can deduce that the Poincar\'e map $\Pi(\rho)$ is analytic at $\rho=0$ because the determining polynomial associated to the branches $\Pi^*(\rho) = \eta \rho + \cdots$ is $\mathcal{P}(\eta) = c (\eta^{m-1}-1)$ and has $\eta=1$ as a simple zero since $\mathcal{P}'(1) = c(m-1) \neq 0$.
\newline

Now we shall prove statement (ii). First we recall that, assuming $\Pi(\rho)$ is analytic at $\rho=0$ for the hole family $\mathcal{X}|_E$, the cyclicity of the origin with respect to perturbation within the family is at most $\# \mathcal{B} - 1$, being $\# \mathcal{B}$ the cardinality of a minimal basis of $\mathcal{B}$. This cyclicity bound can be proved using a Rolle's Theorem kind of argument, see for example Lemma 6.1.6 and Theorem 6.1.7 in \cite{RS}. In our situation $\# \mathcal{B} = 1$ and this proves statement (ii).
\newline

To prove statement (iii) we notice that when $m=1$ the vector field $\mathcal{Y} = f(\rho) \partial_\rho + f(\Pi) \partial_\Pi$ associated to the fundamental equation \eqref{fundamental} has the form $\mathcal{Y} = (\rho+ \cdots) \partial_\rho + (\Pi+ \cdots) \partial_\Pi$. Then there exists an analytic tangent to the identity linearizing change of variables $(\rho, \Pi) \mapsto (\xi, \eta) = \psi(\rho, \Pi) = (\rho + \cdots, \Pi + \cdots)$ defined in a neighborhood of the origin such that $\psi_*\mathcal{Y} = \xi \partial_\xi + \eta \partial_\eta$ with associated differential equation $d \eta / d \xi = \eta/ \xi$, hence with solutions $\eta = c \, \xi$ being $c$ an arbitrary constant. Going back to the original variables, the general solution of \eqref{fundamental} can be found implicitly by $F(\rho, \Pi) = (\Pi + \cdots) - c (\rho + \cdots) = 0$. Since $F$ is analytic at the origin, $F(0,0)=0$ and $\partial_\Pi F(0,0) = 1 \neq 0$, applying the Implicit Function Theorem, we have that for any fixed $c$ there is a unique analytic solution $\Pi^*(\rho)$ of \eqref{fundamental}. From the first term in the Taylor expansion of $\Pi^*$ at $\rho=0$ we see that the Poincar\'e map $\Pi(\rho)$ (which we knew has a formal expansion) corresponds with the choice $c = \eta_1$.
\end{proof}

\section{Apendix}

\subsection{Invariant branches of $\mathcal{X}$ at singular points} \label{Sec-IB2}

We consider a real analytic planar vector field $\mathcal{X} = P(x,y) \partial_x + Q(x,y) \partial_y$ in a neighborhood of a singularity located at the origin. We say that a Puiseux series of the form
\begin{equation}\label{branch-first}
y_i^*(x) = \alpha_0 x^{k_1/k_2} + o(x^{k_1/k_2})
\end{equation}
with $\alpha_0 \in \mathbb{C} \backslash \{ 0 \}$ and $0 < k_1/k_2 \in \mathbb{Q}$ is an {\it invariant branch} of $\mathcal{X}$ at the origin if it satisfies term by term the differential equation $P(x,y) dy - Q(x,y) dx = 0$ associated to $\mathcal{X}$, that is,
\begin{equation}\label{eq-inv-branch}
P(x,y_i^*(x)) \frac{d y_i^*}{dx}(x) - Q(x, y_i^*(x)) = 0,
\end{equation}
for all $x$ in a half-neighborhood of $x=0$. Notice that any invariant branch of $\mathcal{X}$ at the origin must be complex-valuated (non-real) when the origin is a monodromic singularity of $\mathcal{X}$. In particular the curve $y - y^*_i(x) = 0$ (and also the complex conjugated curve $y - \bar{y}^*_i(x) = 0$) are invariant for the real vector field $\mathcal{X}$.

To determine the allowed rational leading exponents $k_1/k_2$ and leading coefficient $\alpha_0$ of the eventual invariant branches one needs to look at the {\it Newton diagram} $\mathbf{N}(\mathcal{X})$ of $\mathcal{X}$, hence we shortly define it. Taking the Taylor developments of $P$ and $Q$ at $(0,0)$ given by
$$
P(x,y) = \sum_{(i,j) \in \mathbb{N}^2} a_{ij} x^i y^{j-1}, \ \ Q(x,y) = \sum_{(i,j) \in \mathbb{N}^2} b_{ij} x^{i-1} y^{j},
$$
the support of $\mathcal{X}$ is defined as ${\rm supp}(\mathcal{X}) = \{ (i, j) \in \mathbb{N}^2 :  (a_{ij}, b_{ij}) \neq (0, 0)\}$ and next we consider the boundary of the convex hull of the set $\bigcup_{(i,j) \in {\rm supp}(\mathcal{X})} \{ (i,j) + \mathbb{R}_+^2 \}$ that
contains a polygon. In the monodromic scenario, the Newton diagram $\mathbf{N}(\mathcal{X})$ of $\mathcal{X}$ is composed by the  edges of that polygon. These edges have endpoints in the set $\mathbb{N}^2$ and, therefore, we can associate to each edge the {\it weights} $(p,q) \in \mathbb{N}^2$, with $p$ and $q$ coprimes, determined by the tangent $q/p$ of the angle between the considered segment and the ordinate axis. Throughout this work we will denote by $W(\mathbf{N}(\mathcal{X})) \subset \mathbb{N}^2$ the set whose elements are all the weights associated to the edges in $\mathbf{N}(\mathcal{X})$.

Once we pick up the weights $(p,q) \in W(\mathbf{N}(\mathcal{X}))$ we fix an expansion \eqref{campo-X} given by $\mathcal{X} = \sum_{j \geq r} \mathcal{X}_j$ where $r \geq 1$ and $\mathcal{X}_j = P_{p+j}(x,y) \partial_x + Q_{q+j}(x,y) \partial_y$ are $(p,q)$-quasihomogeneous vector fields of degree $j$. It can be shown that, for this $(p,q) \in W(\mathbf{N}(\mathcal{X}))$, the leading term $\alpha_0 x^{k_1/k_2}$ of any invariant branch \eqref{branch-first} must be a solution of the differential equation $P_{p+r}(x,y) dy - Q_{q+r}(x, y) dx = 0$ with $k_1/k_2 = q/p$ and $\alpha_0$ a root of a polynomial $\mathcal{Q}(\eta)$ defined by
\begin{equation}\label{det-pol-2}
P_{p+r}(x, \eta x^{q/p}) \frac{d}{d x}(\eta x^{q/p}) - Q_{q+r}(x, \eta x^{q/p}) = \mathcal{Q}(\eta) x^\frac{r+q}{p}.
\end{equation}
In particular, from \eqref{det-pol-2} we see that $y^p-\alpha_0 x^q = 0$ is an irreducible invariant algebraic curve of the leading vector field $\mathcal{X}_r$, see \eqref{campo-X}.

\begin{remark} \label{Demina-remark}
{\rm There is a general method, see \cite{Br,Bru,De}, to know which is the index $n$ of an invariant branch $y^*$ and here we only sketch a part of it using our notation. To each $(p,q) \in W(\mathbf{N}(\mathcal{X}))$ we have the $(p,q)$-quasihomogeneous expansion $\mathcal{X} = \mathcal{X}_r + \cdots$ given in \eqref{campo-X} with
\[
\mathcal{X}_r = P_{p+r}(x,y) \partial_x + Q_{q+r}(x,y) \partial_y
\]
and we associate the dominant balance $E_0[y(x), x] = P_{p+r}(x,y) y' - Q_{p+r}(x,y)$ where the prime indicates derivative with respect to $x$. This dominant balance satisfies $E_0[\alpha_0 x^{q/p}, x] \equiv  0$ where $y^*(x)= \alpha_0 x^{q/p} + \cdots$. Now we calculate the formal G\^{a}teaux derivative of the dominant balance at $y(x) = \alpha_0 x^{q/p}$, that is,
\begin{eqnarray*}
\frac{\delta E_0}{\delta y}[\alpha_0 x^{q/p}] &=& \lim_{s \to 0} \frac{E_0[\alpha_0 x^{q/p} + s x^{q/p+j},x] - E_0[\alpha_0 x^{q/p}, x]}{s} = \Xi(j) x^{\beta},
\end{eqnarray*}
for certain exponent $\beta(j, p, q)$. The zeros of $\Xi(j)$ are called the {\it Fuchs indices} of the dominant balance (also called resonances or Kovalevskaya exponents). In the particular case that there are no Fuchs indices in $\mathbb{Q}^+ \backslash \mathbb{N}$ (non-natural positive rationals numbers) then the branch $y^*$ has index $n = p$. }
\end{remark}

\subsection{Critical parameters and curves of zero angular speed}

From its definition \eqref{crit-set} we notice that $\Theta^{-1}(0) \neq \emptyset$ because $\Omega_{pq} \neq \emptyset$ and $(\varphi, \rho) = (\varphi^*, 0) \in \Theta^{-1}(0)$ for all $\varphi^* \in \Omega_{pq}$. Of course the set $\Theta^{-1}(0)$ is parameter-dependent and it may happen that, for parameters in $\Lambda_{pq} \subset \Lambda$, some curve (branch) emerges from a point $(\varphi^*, 0)$ and lies in $\Theta^{-1}(0)$. The opposite case occurs when the parameter belong to $\Lambda \backslash \Lambda_{pq}$ so that $\Theta^{-1}(0) \backslash \{ \rho = 0\} = \emptyset$. Therefore, for each weight $(p,q) \in W(\mathbf{N}(\mathcal{X}))$, and recalling the definition \eqref{critical-L} of the $(p,q)$-critical parameters, it may occurs that $\Lambda_{pq} = \emptyset$ and also that $\Lambda_{pq} = \Lambda$. For example, in \cite{Ga-Li-Ma-Ma} it is proved that $\Lambda_{11} = \emptyset$ when $\mathcal{X}$  is polynomial and defined by the sum of two homogeneous vector fields. Moreover, defining the {\it critical parameters} $\Lambda^* \subset \Lambda$ as the intersection
$$
\Lambda^* = \bigcap_{(p,q) \in W(\mathbf{N}(\mathcal{X}))} \Lambda_{pq},
$$
another possibility that covers all $\Lambda$ is that $\Lambda^* = \emptyset$ although $\Lambda_{pq} \neq \emptyset$ for all $(p,q) \in W(\mathbf{N}(\mathcal{X}))$.

\begin{remark}
{\rm Looking at the first terms in the expansion
\begin{equation}\label{expansion-Theta}
\Theta(\varphi, \rho) = G_r(\varphi) + G_{r+1}(\varphi) \rho + G_{r+2}(\varphi) \rho^{2} + O(\rho^3),
\end{equation}
we may write down sufficient conditions that guarantee $\Lambda_{pq} = \emptyset$. By monodromy we assume $G_r$ is positive semi-definite in $\mathbb{S}^1$ without loss of generality. Then the condition $\Lambda_{pq} = \emptyset$ holds if one of the following sentences are satisfied for all $\varphi \in J_\varepsilon = \cup_{i=1}^\ell (\varphi_i^*-\varepsilon, \varphi_i^*+\varepsilon)$ with $\varepsilon > 0$ sufficiently small:
\begin{itemize}
  \item $G_{r+1}(\varphi) \geq 0$ (with the linear approximation of $\Theta$);
  \item $G_{r+1}^2(\varphi) - 4 G_{r}(\varphi) G_{r+2}(\varphi) < 0$ (with the quadratic approximation of $\Theta$).
  \item We also can use the cubic approximation of $\Theta$, the discriminant of the cubic, and the Descartes's rule of signs to count how many roots are real positive or negative.
\end{itemize} }
\end{remark}

The point $(\varphi, \rho) = (\varphi^*, 0)$ with $\varphi^* \in \Omega_{pq}$ is a singularity of the map $\Theta$ because $\Theta(\varphi^*, 0) = \partial_\varphi \Theta(\varphi^*, 0) = 0$ since $G_r'(\varphi^*) = 0$ under monodromy, see \cite{Ga-Gi}. Therefore a branch of real positive solutions of the equation $\Theta(\varphi, \rho) = 0$ may bifurcate from the point $(\varphi, \rho) = (\varphi^*, 0)$. In that case there is continuous real-valued function $\rho^*(\varphi)$ defined in a half-neighborhood of $\varphi^*$ such that $\rho^*(\varphi^*) = 0$, $\Theta(\varphi, \rho^*(\varphi)) \equiv 0$, and
\begin{equation}\label{positive-rho}
\rho^*(\varphi) \geq 0.
\end{equation}
Condition $\Lambda_{pq} = \emptyset$  holds if and only if such real positive branches do not exist for any $\varphi^* \in \Omega_{pq}$. In this sense, $\Lambda_{pq} = \emptyset$ is a computable condition applying branching theory based on Newton diagram $\mathbf{N}(\Theta)$ of the analytic function map $\Theta$ at the singularities $(\varphi^*, 0)$. After a translation we put $\varphi^*= 0$ without lost of generality. Any branch (real or complex) emerging from $(\varphi, \rho) = (0, 0)$ can be locally expressed as convergent Puiseux series determined by the descending sections of $\mathbf{N}(\Theta)$ giving $\rho^*(\varphi) = \alpha_0 \varphi^{k_1/k_2} + o(\varphi^{k_1/k_2})$ with $\alpha_0 \in \mathbb{C} \backslash \{ 0 \}$ and $0 < k_1/k_2 \in \mathbb{Q}$. The leading coefficient $\alpha_0$ is a nonzero root of some computable determining polynomial $P(\eta)$ for the particular descending segment of $\mathbf{N}(\Theta)$ with negative slope $-k_1/k_2 \in \mathbb{Q}^-$. The Newton-Puiseux algorithm is used to determine also the first higher order terms of the branch $\rho^*(\varphi)$. More specifically, we say that a branch is {\it simple} if its leading coefficient $\alpha_0$ is a simple root of $P(\eta)$. The Puiseux series of a simple branch is $\rho^*(\varphi) = \sum_{j \geq 0} \alpha_i \varphi^{\frac{k_1+j}{k_2}}$ and it can be proved that if $\alpha_0 \in \mathbb{R}$ then all the coefficients $\alpha_i \in \mathbb{R}$ for all $i \in \mathbb{N}$ so that the branch is real, see for example \cite{VT} for a proof.

By Newton-Puiseux Theorem (see \cite{BK} for instance) there exists a finite factorization
\begin{equation}\label{factor-Puiseux-Theta}
\Theta(\varphi, \rho)  = u(\varphi, \rho)  \prod_{i} (\rho - \rho^*_i(\varphi))^{\kappa_i}
\end{equation}
where $u$ is an analytic unit $u(0,0) \neq 0$, and the $\rho^*_i(\varphi)$ the branches emerging from $(\varphi, \rho) = (0,0)$. Condition $\Lambda_{pq} = \emptyset$ is equivalent to say that, to each singularity $(\varphi^*, 0)$ with $\varphi^* \in \Omega_{pq}$, all the emerging branches are either complex or negative. A sufficient condition for this to happen is that all the leading coefficients $\alpha_0$ associated to each descending segment of $\mathbf{N}(\Theta)$ lies in $\mathbb{C} \backslash \mathbb{R}$, for all $\varphi^* \in \Omega_{pq}$. Of course, if some leading coefficient $\alpha_0$ is a multiple real root of $P(\eta)$ it also may occur that the associated branch be complex. Recall that, even when a branch $\rho^*(\varphi)$ is real next we need to check condition \eqref{positive-rho}.

We say that $\Theta(\varphi, \rho)$ and $\hat\Theta(\varphi, \rho)$ are equivalent if there exist a local diffeomorphism of $\mathbb{R}^2$ of the form $(\varphi, \rho) \mapsto (\alpha(\varphi, \rho), \beta(\rho))$ mapping
the origin to $(\varphi, \rho) = (\varphi^*, 0)$ and a positive function $U(\varphi, \rho)$ such that
\begin{equation}\label{equivalency}
U(\varphi, \rho) \, \Theta(\alpha(\varphi, \rho), \beta(\rho)) = \hat\Theta(\varphi, \rho)
\end{equation}
and where the diffeomorphism preserves the orientations of $\varphi$ and $\rho$, that is, the derivatives $\partial_\varphi \alpha(\varphi, \rho) > 0$ and $\partial_\rho \beta(\rho) > 0$. We call $\hat\Theta$ a normal form of $\Theta$ when it is the simplest representative from a whole equivalence class of mappings. In the particular case that $\rho$ is not changed, that is when $\beta(\rho) = \rho$, then $\Theta(\varphi, \rho)$ and $\hat\Theta(\varphi, \rho)$ are called strongly equivalent.

For each $\rho$ close to zero, let $N_\Theta(\rho)$ be the number of zeros of $\Theta(., \rho)$ locally in some neighborhood of $\varphi^*$. One of the most important consequences of the equivalence of $\Theta$ and $\hat\Theta$ is that $N_\Theta(\rho) = N_{\hat\Theta}(\beta(\rho))$. Using these ideas of singularity theory of maps, the following lemma is proved in Proposition 9.2 of chapter II, page 95 of \cite{GS}. We state it using pur notation.

\begin{lemma}\label{non-elementary}
Let $\Theta(\varphi, \rho)$ be the analytic map with expansion \eqref{expansion-Theta} about $\rho=0$ and the point $(\varphi, \rho) = (\varphi^*, 0)$ with $\varphi^* \in \Omega_{pq}$ a singularity of $\Theta$. If the function $G_{r+1}$ has a simple zero at $\varphi^*$ then the following normal form $\hat\Theta$ is characterized under strongly equivalency: $\hat\Theta(\varphi, \rho) = \delta_1 \varphi^k + \delta_2 \varphi \rho$ where $k \geq 3$ is the multiplicity of $G_r$ at $\varphi^*$ and
$$
\delta_1 = {\rm sgn}\left( \frac{d^{k} G_r}{d \varphi^{k}}(\varphi^*) \right) \neq 0, \ \ \delta_2 = {\rm sgn}\left(G'_{r+1}(\varphi^*) \right) \neq 0.
$$
\end{lemma}

The following result is based on the classification of all the singularities of codimension three or less for the map $\Theta$ at the singularity $(\varphi, \rho) = (\varphi^*, 0)$ with $\varphi^* \in \Omega_{pq}$ which lead to elementary bifurcation problems. The notation in the following classification lemma is the following one: (i) $D^2 \Theta$ stands for the Hessian matrix of $\Theta$ at $(\varphi, \rho) = (\varphi^*, 0)$; (ii) By $\Theta_{v}$ we mean a directional derivative along the eigenvector $v$ associated with the zero eigenvalue of the Hessian $D^2 \Theta$, that is, $\partial_v = v_1 \partial_\varphi + v_2 \partial_\rho$ being $v=(v_1, v_2)$  such that $(D^2 \Theta) v = 0$. The following lemma is a summary of the main results of section 2 in chapter IV of \cite{GS} taking into acount the monodromy of the origin of \eqref{DCC-1}.

\begin{lemma}\label{elementary}
Let $\Theta(\varphi, \rho)$ be the analytic map with expansion \eqref{expansion-Theta} about $\rho=0$  and the point $(\varphi, \rho) = (\varphi^*, 0)$ with $\varphi^* \in \Omega_{pq}$ a singularity of $\Theta$. We define $k \geq 2$ as the multiplicity of $G_r$ at $\varphi^*$. If the origin of \eqref{DCC-1} is monodromic, then $k$ is even. Moreover, the singularity is elementary, i.e., ${\rm codim}(\Theta) \leq 3$ if and only if one of the following conditions hold:
\begin{itemize}
\item[(a)] If $k=2$ and $\det(D^2 \Theta) \neq 0$ then $\hat\Theta(\varphi, \rho) = \delta_1(\theta^2 + \delta_2 \rho^2)$ where $\delta_1 = {\rm sgn}(G''_r(\varphi^*))$ and $\delta_2 = {\rm sgn}(\det(D^2 \Theta))$.

\item[(b)] If $k=2$, $\det(D^2 \Theta) = 0$, and $\Theta_{vvv} \neq 0$ then the normal form is $\hat\Theta(\varphi, \rho) = \delta_1 \varphi^2 + \delta_2 \rho^3$ where $\delta_1 = {\rm sgn}(G_r'''(\varphi^*))$ and $\delta_2 = {\rm sgn}(\Theta_{vvv})$.

\item[(c)] If $k=2$, $\det(D^2 \Theta) = \Theta_{vvv} = 0$ and $s := \Theta_{vvv} \, G_r''(\varphi^*) -3 \Theta^2_{vv\varphi} \neq 0$ then the normal form is $\hat\Theta(\varphi, \rho) = \delta_1 \varphi^2 + \delta_2 \rho^4$ where $\delta_1 = {\rm sgn}(G_r''(\varphi^*))$ and $\delta_2 = {\rm sgn}(s)$.

\item[(d)] $k = 4$ and Lemma \ref{non-elementary} applies.
\end{itemize}
\end{lemma}


\end{document}